\documentclass[a4paper,10pt]{amsart}

\usepackage[utf8]{inputenc} % Linux
\usepackage{amsmath}
\usepackage{amsbsy}
\usepackage{amssymb}
\usepackage{amscd,amsthm}
\usepackage{graphicx}
\usepackage[mathcal]{eucal}
\usepackage{verbatim}
\usepackage[english]{babel}
\usepackage[dvigs]{epsfig}
\usepackage{dsfont}
\usepackage{longtable}
\usepackage{marvosym}
\usepackage{anysize}
\usepackage{hyperref}

\newcommand{\e}{\varepsilon}

\newcommand{\C}{\mathds{C}}

\newcommand{\q}{\quad}

\newtheorem{thm}{Theorem}[section]
\newtheorem{lem}[thm]{Lemma}
\newtheorem{kor}[thm]{Corollary}

\newtheorem{prop}[thm]{Proposition}

\theoremstyle{definition}

\theoremstyle{remark}
\newtheorem*{rema}{Remark}

\title{On Robin's inequality}
\author{Christian Axler}
\email{christian.axler@hhu.de}
\date{\today}

\begin{document}

\begin{abstract}
Let $\sigma(n)$ denotes the sum of divisors function of a positive integer $n$. Robin proved that the Riemann hypothesis is true if and only if the inequality $\sigma(n) < e^{\gamma}n \log \log n$ holds for every positive integer $n \geq 5041$, where $\gamma$ is the Euler-Mascheroni constant. In this paper we establish a new family of integers for which Robin's inequality $\sigma(n) < e^{\gamma}n \log \log n$ hold. Further, we establish a new unconditional upper bound for the sum of divisors function. For this purpose, we use an approximation for Chebyshev's $\vartheta$-function and for some product defined over prime numbers.
\end{abstract}

\maketitle
\section{Introduction}
Let $n$ be a positive integer. The arithmetical function $\sigma$ is defined by
\begin{displaymath}
\sigma(n) = \sum_{d \vert n} d
\end{displaymath}
and denotes the sum of the divisors of $n$. Gronwall \cite[p. 119]{gronwall1913} found the maximal order of $\sigma$ by showing that
\begin{equation}
\limsup_{n \to \infty} \frac{\sigma(n)}{n \log \log n} = e^{\gamma}, \tag{1.1} \label{1.1}
\end{equation}
where $\gamma = 0.5772156\ldots$ denotes the Euler-Mascheroni constant. In the proof of \eqref{1.1}, Gronwall used the asymptotic formula
\begin{equation}
\prod_{p \leq x} \frac{p}{p-1} \sim e^{\gamma} \log x \q\q (x \to \infty), \tag{1.2} \label{1.2}
\end{equation}
where $p$ runs over primes not exceeding $x$, which is due to Mertens \cite[p. 53]{mertens1874}. Under the assumption that the Riemann hypothesis is true, Ramanujan 
\cite{ramanujan1997} showed that the inequality
\begin{displaymath}
\frac{\sigma(n)}{n} < e^{\gamma} \log \log n
\end{displaymath}
holds for all sufficiently large positive integers $n$. Robin \cite[Th\'eor\`eme 1]{robin1984} improved Ramanujan's result by showing the following equivalence.

\begin{prop}[Robin] \label{prop101}
The Riemann hypothesis is true \textit{if and only if}
\begin{equation}
\frac{\sigma(n)}{n} < e^{\gamma} \log \log n \q \q (n \geq 5041). \tag{1.3} \label{1.3}
\end{equation}
\end{prop}
 
This criterion on the Riemann hypothesis is called \textit{Robin's criterion} and the inequality \eqref{1.3} is called \textit{Robin's inequality}. Robin's inequality is proved to hold in many cases (see, for instance, Banks et al. \cite{banks2009}, Briggs \cite{briggs2006}, Grytczuk \cite{grytczuk2007}, and Grytczuk \cite{grytczuk2010}), but remains open in general. In particular, Robin’s inequality has been
proven for several the $t$-free integers. Here a positive integer $n$ is called $t$-free if $n$ is not divisible by the $t$th power of any prime number, and $t$-full otherwise. Choie et al. \cite[Theorem 1.5]{choie2007} showed that the inequality \eqref{1.3} does hold for every 5-free integer $n$ with $n \geq 5041$. This result was extended by Sol\'e and Planat \cite[Theorem 10]{sole2012} by showing that the inequality \eqref{1.3} is true for every 7-free integer $n$. Broughan and Trudgian \cite[Theorem 1]{trudgian2015} used the method investigated by Sol\'e and Planat to find that Robin's inequality holds for every 11-free integer $n$ with $n \geq 5041$. Recently, Morrill and Platt \cite[Theorem 2]{morrillplatt} used the same method combined with an enlargement of Briggs' \cite{briggs2006} result to show that the inequality \eqref{1.3} holds for every 20-free integer $n$ with $n \geq 5041$. In this paper, we apply recently improved estimates for the product in \eqref{1.2} to Sol\'e and Planat's method to obtain the following result.

\begin{thm} \label{thm102}
Robin’s inequality \eqref{1.3} holds for every 21-free integer $n$ with $n \geq 5041$.
\end{thm}

However, the ratio $\sigma(n)/n$ cannot be too large. Robin \cite[Th\'eor\`eme 2]{robin1984} used a lower bound for Chebyshev's $\vartheta$-function
\begin{equation}
\vartheta(x) = \sum_{p \leq x} \log p \tag{1.4} \label{1.4}
\end{equation}
where $p$ runs over primes not exceeding $x$, to show that the weaker inequality
\begin{equation}
\frac{\sigma(n)}{n} < e^{\gamma} \log \log n \left( 1 + \frac{0.6483}{(\log \log n)^2} \right) \tag{1.5} \label{1.5}
\end{equation}
holds unconditionally for every integer $n \geq 3$. We note that the constant in \eqref{1.5} is an approximation to $(\sigma(12)/12 - e^{\gamma}\log \log 12)\log \log 12$, so a better constant can be achieved if we consider \eqref{1.5} for $n \geq n_0 > 12$ (see, for instance, \cite[Theorem 1.1]{axlerR}). 
% If the Riemann hypothesis is false, then, by Robin's criterion, the inequality \eqref{1.3} is false when $\sigma(n)/n$ is large.
The advantage of the inequality \eqref{1.5} is that it holds for every positive integer $n$ where $\log \log n$ is positive. Setting
\begin{displaymath}
\mathcal{A} = \{ 1,2,4,5,6,8,9,10,12,16,18,20,24,30,36,48,60,72,84,120,180,240,360,840,2520,5040\},
\end{displaymath}
we obtain the following improvement of \eqref{1.5}.

\begin{thm} \label{thm103}
For every positive integer $n$ satisfying $n \notin \mathcal{A}$, we have
\begin{displaymath}
\frac{\sigma(n)}{n} < e^{\gamma} \log \log n \left( 1 + \frac{0.0094243}{(\log \log n)^3} \right).
\end{displaymath}
\end{thm}

Let $\nu_p(n)$ denote the $p$-adic valuation of an integer $n$. In \cite[Theorem 2]{hertlein}, Hertlein showed that Robin's inequality \eqref{1.3} holds for every integer $n \geq 5041$ satisfying $\nu_2(n) \leq 19$, $\nu_3(n) \leq 12$, $\nu_5(n) \leq 7$, $\nu_7(n) \leq 6$, or $\nu_{11}(n) \leq 5$. Finally, Theorem \ref{thm103} allows us to extend Hertlein's result by establishing the following new family of integers for which Robin's inequality \eqref{1.3} hold.

\begin{thm} \label{thm104}
Robin's inequality \eqref{1.3} holds for every integer $n \geq 5041$ satisfying at least one of the following properties:
\begin{itemize}
 \item $\nu_2(n) \leq 20$,
 \item $\nu_5(n) \leq 8$,
 \item $\nu_p(n) \leq 4$, if $p$ is a prime with $11 < p \leq 19$,
 \item $\nu_p(n) \leq 3$, if $p$ is a prime with $19 < p \leq 41$,
 \item $\nu_p(n) \leq 2$, if $p$ is a prime with $41 < p \leq 139$,
 \item $\nu_p(n) = 1$, if $p$ is a prime with $139 < p \leq 1777$.
 \end{itemize}
\end{thm}

\section{Proof of Theorem \ref{thm102}}

In their common paper \cite{sole2012}, Sol\'e and Planat introduced the multiplicative function $\Psi_t$ defined for any integer $t \geq 2$ by
\begin{displaymath}
\Psi_t(n) = n \prod_{p \vert n} \left( 1 + \frac{1}{p} + \cdots + \frac{1}{p^{t-1}} \right).
\end{displaymath}

\begin{rema}
In the case where $t =2$, this is just the classical Dedekind function $\Psi$.
% which occurs in the theory of modular forms and in analytic number theory.
\end{rema}

In the case where $n$ is $t$-free, they \cite[p.\:302]{sole2012} note the following result.

\begin{lem}[Sol\'e and Planat] \label{lem201}
If $n$ is $t$-free, then $\sigma(n) \leq \Psi_t(n)$.
\end{lem}

We set
\begin{displaymath}
R_t(n) = \frac{\Psi_t(n)}{n \log \log n}.
\end{displaymath}
So, in order to prove Theorem \ref{thm102}, it suffices to show that the inequality 
\begin{displaymath}
R_{21}(n) < e^{\gamma}
\end{displaymath}
holds for every $n \geq 5041$. For this purpose we introduce the $k$th primorial $N_k$ as the product of the first $k$ primes,
\begin{displaymath}
N_k = \prod_{i=1}^k p_i.
\end{displaymath}
Sol\'e and Planat \cite[Proposition 1]{sole2012} proved that the primorial numbers and their multiples are exactly the champion numbers of the function $x \mapsto \Psi_t(x)/x$. Further, they proved that, in order to maximize the funktion $R_t(n)$ it is enough to consider its value at primorial integers by showing the following result.

\begin{lem}[Sol\'e and Planat] \label{lem202}
Let $k$ be an integer satisfying $k \geq 2$. For any integer $n$ with $N_k \leq n < N_{k+1}$, we have $R_t(n) \leq R_t(N_k)$.
\end{lem}

A positive integer $n$ is \textit{colossally abundant} if there is an $\e > 0$ so that $\sigma(n)/n^{1+\e} \geq \sigma(k)/k^{1+\e}$ for every integer $k \geq 2$. Let $M_1$ are $M_2$ consecutive colossally abundant numbers satisfying the inequality \eqref{1.3}. Then, Robin \cite[Proposition 1]{robin1984} showed that Robin's inequality \eqref{1.3} holds for every integer $n$ such that $M_1 \leq n \leq M_2$. Briggs \cite{briggs2006} computed that Robin's inequality \eqref{1.3} holds for every colossally abundant number $n$ with $5041 \leq n \leq 10^{10^{10}}$. Hence, Robin's inequality is fulfilled for every integer $n$ so that $5041 \leq n \leq 10^{10^{10}}$. Recently, Morrill and Platt \cite{morrillplatt} extended Briggs result to every integer $n$ with $5041 \leq n \leq 10^{10^{13.11485}}$. The argument in the proof of the last expansion implies a slightly better result.

\begin{lem}[Morrill, Platt] \label{lem203}
Let $k_0 = 999,999,476,056$. Then $p_{k_0} = 29,996,208,012,611$ and Robin's inequality \eqref{1.3} holds for every integer $n$ so that
\begin{displaymath}
5041 \leq n \leq N_{k_0}.
\end{displaymath}
\end{lem}

\begin{rema}
For further details on colossally abundant numbers see \cite{erdos1944}. %\cite{ramanujan1997}, \cite{erdos1975}, \cite{akbary2009}, \cite{caveney2012} and %\cite{nazardonyavi2014}.
\end{rema}

We proceed to study the expression $R_t(N_k)$. For $s \in \C$ with $\text{Re}(s) > 1$ the Euler product formula for the Riemann zeta function is given by
\begin{displaymath}
\zeta(s) = \prod_{\text{$p$ prime}} \frac{1}{1-p^{-s}}.
\end{displaymath}
Using this formula together with the definition of $\Psi_t$ and the definition \eqref{1.4} of Chebyshev's $\vartheta$-function, we get the following identity for $R_t(N_k)$.

\begin{lem} \label{lem204}
For every integer $t \geq 2$ and every positive integer $k$, we have
\begin{displaymath}
R_t(N_k) = 
% \frac{1}{\log \log N_k} \times \prod_{p \leq p_k} \frac{1-p^{-t}}{1-p^{-1}} = 
\frac{\prod_{p > p_k}(1-p^{-t})^{-1}}{\zeta(t) \log \vartheta(p_k)} \times \prod_{p \leq p_k} \frac{p}{p-1}.
\end{displaymath}
\end{lem}

In the following proof of Theorem \ref{thm102}, we do apply recently improved estimates for Chebyshev's $\vartheta$-function and the product in \eqref{1.2}, respectively, to formula of $R_t(N_k)$ given in Lemma \ref{lem204}.

\begin{proof}[Proof of Theorem \ref{thm102}]
First, let $n$ be an $21$-free integer with $n \geq 5041$ and let $k_0 = 999,999,476,056$. If $n \leq N_{k_0}$, the claim follows directly from Lemma \ref{lem203}. So we can assume that $n \geq N_{k_0}$ and we define $k$ to be the unique integer with $N_k \leq n < N_{k+1}$. If we combine Lemmata \ref{lem201} and \ref{lem202}, it turns out that
\begin{displaymath}
\frac{\sigma(n)}{n \log \log n} \leq R_{21}(N_k).
\end{displaymath}
So it suffices to check that $e^{-\gamma}R_{21}(N_k) < 1$. By \cite[Lemma 6.4]{choie2007}, we have
\begin{displaymath}
\prod_{p > p_k}(1-p^{-21})^{-1} \leq \exp (21/(20p_k^{20})). 
\end{displaymath}
We can apply this inequality to Lemma \ref{lem204} and see that
\begin{displaymath}
e^{-\gamma}R_{21}(N_k)\leq \frac{\exp (-\gamma + 21/(20p_k^{20}))}{\zeta(21) \log \vartheta(p_k)} \times \prod_{p \leq p_k} \frac{p}{p-1}.
\end{displaymath}
Since $k \geq k_0 = 999,999,476,056$, we have $p_k \geq 29,996,208,012,611$. So can use \cite[Proposition 3.4]{axlernew} to get that the inequality
\begin{equation}
e^{-\gamma}R_{21}(N_k)\leq \frac{\log p_k}{\zeta(21) \log \vartheta(p_k)} \times \exp \left( \frac{21}{20p_{k_0}^{20}} + \frac{0.024334}{3 \log^3 p_{k_0}} \left( 1 + \frac{15}{4\log p_{k_0}} \right) \right) \tag{2.1} \label{2.1}
\end{equation}
holds. To estimate the expression $\log p_k/ \log \vartheta(p_k)$, we use \cite[Proposition 1.1]{axlernew} to obtain that
\begin{displaymath}
\log \vartheta(p_k) \geq a\log p_k, 
\end{displaymath}
where $a = 1-0.024334/\log^3 p_{k_0}$. The function $f(x) = \log x/\log(ax)$ is strictly decreasing for $x \geq p_{k_0}$ and it follows that $\log p_k/\log \vartheta(p_k) \leq \log p_{k_0}/\log(ap_{k_0})$. Applying this to \eqref{2.1}, we get
\begin{displaymath}
e^{-\gamma}R_{21}(N_k)\leq \frac{\log p_{k_0}}{\zeta(21) \log(ap_{k_0}))} \times \exp \left( \frac{21}{20p_{k_0}^{20}} + \frac{0.024334}{3 \log^3 p_{k_0}} \left( 1 + \frac{15}{4\log p_{k_0}} \right) \right).
\end{displaymath}
Since the right-hand side of the last inequality is less than 1, we arrived at the end of the proof.
\end{proof}

Note that every 21-free integer $n$ fulfills $n \geq 5041$. So we can reformulate Robin's criterion (see Proposition \ref{prop101}) as follows.

\begin{kor}
The Riemann hypothesis is true if and only if Robin's inequality \eqref{1.3} holds for all 21-full integers.
\end{kor}

\begin{rema}
One form of the Prime Number Theorem states that $\vartheta(x) = x + o(x)$ as $x \to \infty$. If we apply this and \eqref{1.2} to Lemma \ref{lem204}, we see that for every $t \geq 2$ the asymptotic formula
\begin{displaymath}
\lim_{k \to \infty} R_t(N_k) = \frac{e^{\gamma}}{\zeta(t)}
\end{displaymath}
holds (see also \cite[Proposition 3]{sole2012}). This means that for every $t \geq 2$ there is a positive integer $k_0 = k_0(t)$ so that $R_t(N_k) < e^{\gamma}$ for every $k \geq k_0$ which implies that for every $t \geq 2$ Robin’s inequality \eqref{1.3} holds for every $t$-free integer $n$ with $n \geq N_{k_0}$.
\end{rema}

\section{Proof of Theorem \ref{thm103}}

Let $\varphi$ be Euler's totient function. Since $\varphi$ is multiplicative and fulfills $\varphi(p^k) = p^k(1-1/p)$ for any prime $p$ and any positive integer $k$, we get
\begin{displaymath}
\varphi(n) = n\prod_{p \vert n}\left( 1 - \frac{1}{p} \right)
\end{displaymath}
for every positive integer $n$. Let $n = q_1^{e_1} \cdot \ldots \cdot q_k^{e_k}$, where $q_i$ are primes and $e_i \geq 1$. Note that the function $\sigma$ is multiplicative and satisfies for any prime number $p$ and any nonnegative integer $m$ the identity
\begin{displaymath}
\sigma(p^m) = 1 + p + \cdots + p^{m-1} + p^m = \frac{p^{m+1}-1}{p-1}.
\end{displaymath}
Then, it is easy to show (see, for example, \cite[Lemma 2]{grytczuk2007}) that $\sigma$ and $\varphi$ are connected by the identity
\begin{equation}
\frac{\sigma(n)}{n} = \frac{n}{\varphi(n)} \prod_{i=1}^k \left( 1 - \frac{1}{q_i^{1 + e_i}} \right), \tag{3.1} \label{3.1}
\end{equation}
which implies the inequality
\begin{equation}
\frac{\sigma(n)}{n} < \frac{n}{\varphi(n)}. \tag{3.2} \label{3.2}
\end{equation}
Now we give a proof of Theorem \ref{thm103} in which the inequality \eqref{3.2} and effective estimates for Chebyshev's $\vartheta$-function, obtained by Broadbent et al. \cite[Table 15]{kadiri}, play an important role.

\begin{proof}[Proof of Theorem \ref{thm103}]
First, let $n$ be a positive integer with $n \geq 5041$ and let $k_0$ be a positive integer given by $k_0 = \pi(29,996,161,880,813)$. If $n \leq N_{k_0}$, the claim follows directly from Lemma \ref{lem202}. So we can assume that $n \geq N_{k_0}$. We define $k$ to be the unique integer with $N_k \leq n < N_{k+1}$. Since $p_k \geq e^{31.03}$, we can use \cite[Table 15]{kadiri} to see that
\begin{equation}
\log \log N_k = \log \vartheta(p_k) > \log p_k  + \log \left( 1 - \frac{3.3277 \times 10^{-4}}{\log^2 p_k} \right). \tag{3.3} \label{3.3}
\end{equation}
Let $a_0 = 0.0094243$. Note that the function $x \mapsto x + a_0/x^2$ is a strictly increasing function on $(\sqrt[3]{2a_0}, \infty)$. If we combine this remark and the fact that $n \geq N_k$ with \eqref{3.3}, it turns out that
\begin{displaymath}
\log \log n + \frac{a_0}{(\log \log n)^2} > \log p_k \left( 1  + \frac{a_0}{\log^3 p_k} \right) + \log \left( 1 - \frac{3.3277 \times 10^{-4}}{\log^2 p_k} \right).
% \log \log N_k + \frac{a_0}{(\log \log N_k)^2} > \log p_k  + \log \left( 1 - \frac{3.3277 \times 10^{-4}}{\log^2 p_k} \right) + \frac{a_0}{\log^2p_k }.
\end{displaymath}
The last inequality implies that
\begin{displaymath}
\log \log n + \frac{a_0}{(\log \log n)^2} > \log p_k \exp \left( \frac{0.024334}{3\log^3 p_k} \left(1+\frac{15}{4\log p_k} \right) \right).
% \log \log N_k + \frac{a_0}{(\log \log N_k)^2} > \log p_k  + \log \left( 1 - \frac{3.3277 \times 10^{-4}}{\log^2 p_k} \right) + \frac{a_0}{\log^2p_k }.
\end{displaymath}
Now we can use \cite[Proposition 3.4]{axlernew} to get
\begin{equation}
e^{\gamma}\log \log n + \frac{e^{\gamma}a_0}{(\log \log n)^2} > \prod_{p \leq p_k} \left( 1 - \frac{1}{p} \right)^{-1} = \frac{N_k}{\varphi(N_k)}. \tag{3.4} \label{3.4}
\end{equation}
If we combine \eqref{3.2} with the fact that
\begin{equation}
\frac{N_k}{\varphi(N_k)} \geq \frac{n}{\varphi(n)},\tag{3.5} \label{3.5}
\end{equation}
we see that $N_k/\varphi(N_k) \geq \sigma(n)/n$. Applying this inequality to \eqref{3.4}, we obtain the required inequality
% \begin{displaymath}
% e^{\gamma}\log \log N_k + \frac{e^{\gamma}a_0}{(\log \log N_k)^2} > \frac{\sigma(n)}{n}
% \end{displaymath}
for every $n \geq N_{k_0}$. Finally, we check the required inequality every positive integer $n \leq 5040$ with $n \notin \mathcal{A}$.
\end{proof}

\begin{rema}
At this point we would like to point out a small error in \cite{axlerR}. If we set
\begin{displaymath}
\mathcal{B} = \{ 1,2,4,6,8,10,12,16,18,20,24,30,36,48,60,72,120,180,240,360,2520\},
\end{displaymath}
the present author \cite[Thoerem 1.1]{axlerR} claimed that the waeker inequality
\begin{equation}
\frac{\sigma(n)}{n} < e^{\gamma} \log \log n \left( 1 + \frac{0.1209}{(\log \log n)^3} \right) \tag{3.6} \label{3.6}
\end{equation}
holds for every positive integer $n$ with $n \notin \mathcal{B}$. However, a direct computation (or Theorem \ref{thm102}) shows that the inequality \eqref{3.6} does not hold for $n=5040$ either. So we need to replace $\mathcal{B}$ with $\mathcal{B} \cup \{ 5040 \}$ in the above result.
\end{rema}

\begin{rema}
Under the assumption that the Riemann hypothesis is true, Ramanujan \cite[p.\:143]{ramanujan1997} gave the asymptotic upper bound
\begin{displaymath}
\frac{\sigma(n)}{n} \leq e^{\gamma} \left( \log \log n - \frac{2(\sqrt{2}-1)}{\sqrt{\log n}} + S_1(\log n) + \frac{O(1)}{\sqrt{\log n} \log \log n} \right)
\end{displaymath}
with $S_1(x) = \sum_{\rho} x^{\rho-1}/|\rho|^2$, where $\rho$ runs over the non-trivial zeros of the Riemann $\zeta$ function. Nicolas \cite[Theorem 1.1 and Corollary 1.2]{nicolas} found some effective forms of this asymptotic result.
\end{rema}

In a different direction, Ivi\'c \cite[Theorem 1]{ivic1977} showed that the inequality
\begin{equation}
\frac{\sigma(n)}{n} < 2.59 \log \log n \tag{3.7} \label{3.7}
\end{equation}
holds for every integer $n \geq 7$.
After some improvements of the constant in \eqref{3.7}, see for example \cite{robin1984}, \cite{grytczuk2007}, and \cite{akbary2007}, the currently best such inequality was found by Hertlein \cite[Theorem 4]{hertlein}.
% Robin \cite[Proposition 2, p. 210]{robin1984} refined this result by proving that
% \begin{displaymath}
% \frac{\sigma(n)}{n} < 
% % \frac{\sigma(12)}{12 \log \log 12} \log \log n \leq 
% 2.5635 \log \log n
% \end{displaymath}
% for every integer $n \geq 7$. Grytczuk \cite[Remark 5]{grytczuk2007} showed that
% \begin{displaymath}
% \frac{\sigma(n)}{n} < 1.076 e^{\gamma} \log \log n
% \end{displaymath}
% for every integer $n > 3^9$. A further improvement was found by Akbary, Friggstad and Juricevic \cite[Theorem 1.3]{akbary2007}. They showed 
% that
% \begin{displaymath}
% \frac{\sigma(n)}{n} < 
% % \frac{\sigma(180)}{180 \log \log 180} \log \log n \leq 
% 1.0339 e^{\gamma} \log \log n
% \end{displaymath}
% for every integer $n \geq 121$. The currently best such inequality was found by Hertlein \cite[Theorem 4]{hertlein2016}.
He proved that the inequality
\begin{displaymath}
\frac{\sigma(n)}{n} < (1+5.645 \times 10^{-7})e^{\gamma} \log \log n
\end{displaymath}
holds for every integer $n \geq 5041$. Using Theorem \ref{thm103} and Lemma \ref{lem203}, we find the following refinement.

\begin{kor} \label{kor104}
For every positive integer $n$ with $n \notin \mathcal{A} \cup \{ 3, 720\}$, we have
\begin{displaymath}
\frac{\sigma(n)}{n} < (1+3.15367 \times 10^{-7})e^{\gamma} \log \log n.
\end{displaymath}
\end{kor}

\begin{proof}
Let $a_0 = 0.0094243$ and let $\e$ be a positive real number.
% satisfying $\e < \e_0$, where $\e_0 = 5.645 \times 10^{-7}$.
From Theorem \ref{thm103}, it follows directly that the inequality
\begin{equation}
\frac{\sigma(n)}{n} < (1+\e)e^{\gamma} \log \log n \tag{3.8} \label{3.8}
\end{equation}
holds for every positive integer $n \notin \mathcal{A}$ with $n \geq \exp(\exp(\sqrt[3]{a_0/\e}))$. If we set $\e = 3.15367 \times 10^{-7}$, we get that the inequality \eqref{3.8} holds for every integer $n$ with $n \geq N_{k_0}$, where $k_0 = 999,999,476,056$. For every integer $n$ with $5041 \leq n \leq N_{k_0}$, the result follows directly from Lemma \ref{lem203}. For smaller values of $n$, we verify the required inequality with a computer.
\end{proof}

\section{Proof of Theorem \ref{thm104}}

Finally we can use Lemma \ref{lem203} and Theorem \ref{thm103} to give the following proof of Theorem \ref{thm104}.

\begin{proof}[Proof of Theorem \ref{thm104}]
We only show that Robin's inequality \eqref{1.3} holds for every integer $n \geq 5041$ with $\nu_2(n) \leq 20$. The proof of the remaining cases is quite similar and we leave the details to the reader. Let $a_0 = 0.0094243$ and $k_0 = 999,999,476,056$. By Lemma \ref{lem203}, it suffices to consider the case where $n \geq N_{k_0}$. If $\nu_2(n) \leq 20$, we can combine \eqref{3.1}, \eqref{3.4}, and \eqref{3.5} to see that the inequality
\begin{displaymath}
\frac{\sigma(n)}{n} < e^{\gamma} \left( 1 - \frac{1}{2^{21}} \right) \left( \log \log n + \frac{a_0}{(\log \log n)^2} \right) \tag{4.1} \label{4.1}
\end{displaymath}
holds. Since $n \geq N_{k_0} \geq \exp(\exp(\sqrt[3]{2^{21}-1}))$, we obtain that the inequality \eqref{4.1} implies Robin's inequality \eqref{1.3} for every integer $n \geq N_{k_0}$ with $\nu_2(n) \leq 20$ and we arrive at the end of the proof.
\end{proof}

% \section*{Acknowledgements}
% I wish to thank the anonymous referee for his or her careful reading of the paper and their helpful comments.

\end{document}